\documentclass{amsart}

\usepackage{amsmath,amssymb,amsthm,mathtools}
\usepackage[initials]{amsrefs}
\usepackage{enumitem,microtype}
\usepackage[dvipsnames]{xcolor}
\usepackage[colorlinks=true,linkcolor=blue!55!black,citecolor=green!45!black,urlcolor=red!60!black]{hyperref}

\numberwithin{equation}{section}
\usepackage{fourier}
\usepackage[margin=1.5in]{geometry}

\newtheorem{theorem}{Theorem}[section]
\newtheorem{lemma}[theorem]{Lemma}
\newtheorem{proposition}[theorem]{Proposition}
\newtheorem{corollary}[theorem]{Corollary}
\newcommand{\Ent}{\operatorname{Ent}}
\newcommand{\KL}{\operatorname{KL}}
\newcommand{\E}{\mathbb E}
\newcommand{\dd}{\,\mathrm d}
\newcommand{\tr}{\operatorname{tr}}
\newcommand{\Id}{I_d}

\newcommand{\COT}{\mathcal C_{\rm OT}}

\title[Entropy hypocoercivity and query complexity for BPS]
      {Entropy hypocoercivity and query complexity\\
       for the bouncy particle sampler}
\author{Jianfeng Lu}
\address{Department of Mathematics, Department of Physics, and Department of Chemistry, Duke University}
\email{jianfeng@math.duke.edu}
\author{Yinchen Luo}
\address{Department of Mathematics, Duke University}
\email{yinchen.luo@duke.edu}
\date{}
\hypersetup{
  pdftitle={Entropy hypocoercivity and query complexity for the bouncy particle sampler},
  pdfauthor={Jianfeng Lu and Yinchen Luo}
}

\begin{document}

\begin{abstract}
We study convergence in relative entropy of the bouncy particle sampler with Gaussian velocity refreshment, for a target \(e^{-U}\) on \(\mathbb R^d\) with \(U\) \(m\)-strongly convex and \(L\)-smooth, and condition number \(\kappa=L/m\). For initial laws with bounded density relative to equilibrium, we show that the relative entropy decays exponentially up to an arbitrary additive error \(\delta\), with a rate that depends only logarithmically on the density bound and on \(\delta\). From a Gaussian cold start and with a suitably chosen refreshment rate, the relative entropy falls below \(\delta\) after time \(\widetilde O(\sqrt{d/m})\). Combined with exact simulation by windowed thinning, this gives an expected gradient-query complexity of \(\widetilde O(d\sqrt\kappa)\), improving the \(\widetilde O(d^2\sqrt\kappa)\) bound based on \(L^2\) hypocoercivity. 
\end{abstract}

\maketitle

\section{Introduction and main results}\label{sec:setting}
We prove quantitative bounds on relative entropy decay and gradient-query complexity for the bouncy particle sampler (BPS) with Gaussian velocity refreshment. We consider an \(m\)-strongly convex, \(L\)-smooth potential \(U:\mathbb R^d\to\mathbb R\) and write \(\kappa=L/m\) for its condition number. For initial densities bounded relative to equilibrium, we obtain an entropy estimate consisting of an exponentially decaying term and an explicit additive remainder. The estimate holds for every constant refreshment rate. The remainder can be made arbitrarily small at the cost of a slower decay rate.

From a Gaussian cold start, we choose a constant refreshment rate for each prescribed accuracy and bound the deterministic time needed to reach that accuracy in relative entropy. Combining this convergence bound with exact simulation by windowed thinning gives an expected gradient-query complexity of \(\widetilde O(d\sqrt\kappa)\), with logarithmic factors in \(d\), \(\kappa\), and inverse accuracy suppressed.

Throughout, we assume \(d\ge1\) and \(U\in C^2(\mathbb R^d)\), with
\[
 m\Id\preceq\nabla^2U\preceq L\Id,
 \qquad 0<m\le L<\infty.
\]
We place the minimizer at the origin and normalize \(U(0)=0\).

Write \(Z=\int_{\mathbb R^d}e^{-U(x)}\dd x\) for the normalizing constant and \(\varphi_d=N(0,\Id)\) for the Gaussian velocity law. The spatial target and joint equilibrium measure are
\[
 \pi_x(\dd x)=Z^{-1}e^{-U(x)}\dd x,
 \qquad \pi=\pi_x\otimes\varphi_d.
\]
Let \(\Pi_v f(x)=\int f(x,v)\dd\varphi_d(v)\) denote velocity averaging. For a constant refreshment rate \(\gamma>0\), BPS has generator
\begin{equation*}
 \begin{aligned}
 \mathcal L f&=v\cdot\nabla_x f+(v\cdot\nabla U)_+
       [f(x,R_xv)-f(x,v)]+\gamma(\Pi_v f-f),\\
 R_xv&=v-2(v\cdot n)n,\qquad n=\nabla U/|\nabla U|.
 \end{aligned}
\end{equation*}
At the minimizer, choose any unit vector for \(n\); the bounce rate is zero there. Each refreshment event replaces the velocity by an independent draw from \(\varphi_d\). The measure \(\pi\) is invariant for this dynamics.

Let \(\mathcal P_t\) denote the BPS transition semigroup. We write \(\mu_t=\mu_0\mathcal P_t\) for the law at time \(t\) and \(g_t=\dd\mu_t/\dd\pi\) for its density relative to \(\pi\). For a probability density \(g\) with respect to \(\pi\), the relative entropy is
\[
 \Ent_\pi(g)=\int g\log g\dd\pi=\KL(g\pi\|\pi).
\]

\begin{samepage}
\begin{theorem}[Entropy decay with an additive remainder]\label{thm:master}
Suppose \(0\le g_0\le M\) and \(\int g_0\dd\pi=1\).
There is a universal constant \(c>0\) such that, for every constant refreshment rate \(\gamma>0\) and every \(0<\delta\le1/2\), the BPS trajectory satisfies
\begin{equation}\label{eq:master}
 \Ent_\pi(g_t)
 \le 2\exp\!\left[-\frac{c\gamma t}
                   {d+\log M+\log(1/\delta)+\gamma^2/m}\right]
       \Ent_\pi(g_0)+\delta
\end{equation}
for all \(t\ge0\).
\end{theorem}
\end{samepage}

The density bound and the additive remainder are consistent with the obstruction of Monmarch\'e and Wang~\cite{MonmarcheWangEntropy}*{Theorem~2}: even for a standard Gaussian target, BPS admits no multiplicative KL contraction that is uniform over initial laws with small relative entropy. In~\eqref{eq:master}, the rate depends on \(\log M\) and \(\log(1/\delta)\), so the estimate gives no such uniform contraction.

\begin{corollary}[Gaussian cold start]\label{cor:cold}
Suppose BPS is initialized from
\[
 \mu_0=N(0,L^{-1}\Id)\otimes N(0,\Id).
\]
For every \(0<\delta\le1/2\), there is a constant refreshment rate \(\gamma>0\) such that \(\KL(\mu_0\mathcal P_T\|\pi)\le\delta\) at a deterministic time \(T\ge0\) satisfying
\begin{equation}\label{eq:time-order}
 T\le\frac{C}{\sqrt{m}}
 \sqrt{d\log(e\kappa)+\log(1/\delta)}
 \log\!\left(\frac{4(1+d\log\kappa)}{\delta}\right).
\end{equation}
Here \(C\) is a universal constant.
\end{corollary}

Following Lu and Luo~\cite{LuLuo}*{Section 3.3}, let \(Q_\nabla(T)\) denote the expected number of gradient queries used by windowed thinning up to time \(T\).

\begin{corollary}[Expected gradient-query complexity]\label{cor:queries}
Suppose BPS is initialized from the Gaussian cold start in Corollary~\ref{cor:cold}. For every \(0<\delta\le1/2\), there are a constant refreshment rate \(\gamma>0\) and a deterministic horizon \(T\ge0\) such that exact simulation by windowed thinning attains \(\KL(\mu_0\mathcal P_T\|\pi)\le\delta\) with expected gradient-query cost
\begin{equation}\label{eq:query}
 Q_\nabla(T)
 \le C\sqrt{\kappa d}\,
 \sqrt{d\log(e\kappa)+\log(1/\delta)}
 \log\!\left(\frac{4(1+d\log\kappa)}{\delta}\right).
\end{equation}
Here \(C\) is a universal constant.
\end{corollary}

The bound~\eqref{eq:query} is \(\widetilde O(d\sqrt\kappa)\); it relies on the windowed-thinning estimate of Lu and Luo~\cite{LuLuo}*{Proposition 9}. By data processing, the KL guarantees in Corollaries~\ref{cor:cold} and~\ref{cor:queries} also hold for the position marginal.

Following Lu~\cite{LuUnderdampedEntropy}, we add to the entropy a small multiple of the Wasserstein entropy current, which supplies the spatial dissipation that refreshment alone does not provide. For BPS, the time derivative of the current contains an additional defect term, a force correction that vanishes at local velocity equilibrium. We use convexity to express the defect term in terms of the potential difference along the transport, and retain the nonnegative log-determinant defect in the stress estimate of~\cite{LuUnderdampedEntropy} to absorb it. On the region where the absolute potential difference along the transport is at most a cutoff level \(r\), a Gaussian entropy-duality bound and Jensen's inequality control the defect term at a velocity-entropy cost proportional to \(r\). Outside this region, the propagated density bound \(g_t\le M\) makes the contribution exponentially small in \(r\); this is the source of the additive remainder in Theorem~\ref{thm:master}.

\subsection{Related works}
The bouncy particle sampler originated with Peters and de With~\cite{PetersDeWith2012} and was developed into a general-purpose sampler by Bouchard-C\^ot\'e, Vollmer, and Doucet~\cite{BouchardCoteVollmerDoucet}. Harris-type arguments establish its geometric ergodicity under tail and curvature conditions~\cites{DeligiannidisBouchardCoteDoucet2019,DurmusGuillinMonmarche2020}. Bierkens, Kamatani, and Roberts derived scaling limits of BPS in high dimension and for anisotropic targets~\cites{BierkensKamataniRoberts2022,BierkensKamataniRoberts2025}. Agrawal, Bierkens, Kamatani, and Roberts~\cite{AgrawalBierkensKamataniRoberts2025} studied the transient phase of piecewise deterministic samplers through fluid limits.

Explicit convergence rates come from hypocoercivity. Andrieu, Durmus, N\"usken, and Roussel~\cite{AndrieuDurmusNueskenRoussel2021} developed \(L^2\) hypocoercivity for a broad class of piecewise deterministic samplers, including BPS. Lu and Wang~\cite{LuWangPDMP}*{Theorem~1} made the \(L^2\) rate explicit in \(d\), \(m\), and \(L\), building on the analysis of underdamped Langevin dynamics by Cao, Lu, and Wang~\cite{CaoLuWangUnderdamped}. For the potentials considered here and \(\gamma=\sqrt{dm}\), their bound on the BPS rate is of order \(\sqrt{m/d}\). Eberle and L\"orler~\cite{EberleLorlerLifts} interpret BPS as a nonreversible lift of overdamped Langevin dynamics and show that such a lift improves the relaxation time by at most a square-root factor. In relative entropy, Monmarch\'e and Wang~\cite{MonmarcheWangEntropy}*{Theorem~2} showed that BPS admits no uniform multiplicative contraction over initial laws with small relative entropy, even for a standard Gaussian target. The density bound and the additive remainder in Theorem~\ref{thm:master} are consistent with this obstruction; see the discussion after the theorem.

For query complexity, Lu and Luo~\cite{LuLuo} simulate BPS exactly by windowed thinning, with an expected number of gradient queries proportional to the time horizon. Combined with the \(L^2\) rate of Lu and Wang, this gives \(\widetilde O(d^2\sqrt\kappa)\) gradient queries for total-variation accuracy from a Gaussian cold start~\cite{LuLuo}*{Theorem~3}. By Pinsker's inequality, Corollary~\ref{cor:queries} gives the same guarantee with \(\widetilde O(d\sqrt\kappa)\) queries. The gain comes from how the initial law enters. In the \(L^2\) analysis, the logarithm of its \(\chi^2\)-divergence, of order up to \(d\log\kappa\) for the cold start, multiplies the inverse rate \(\sqrt{d/m}\). In Theorem~\ref{thm:master}, the density bound enters only through \(\log M\) in the denominator of the rate, alongside \(d\). This yields the single square-root factor in~\eqref{eq:time-order}.

\subsection*{Organization}
Section~\ref{sec:moments} develops the moment equations and recalls the Wasserstein entropy current. Section~\ref{sec:current} establishes the closed entropy and current inequalities. Section~\ref{sec:closure} then proves Theorem~\ref{thm:master} and Corollaries~\ref{cor:cold} and~\ref{cor:queries}. Appendix~\ref{sec:weak} supplies the regularity and approximation arguments for bounded initial densities.

\subsection*{AI Usage} This research was motivated by trying to reduce the cold-start complexity of our previous work \cite{LuLuo} and to better understand the obstruction to entropy convergence of BPS \cite{MonmarcheWangEntropy}. Through our interactions with GPT 6 Astra, including various human suggestions to adapt the argument of accelerated entropy convergence in \cite{LuUnderdampedEntropy} and \cite{LiLuHypercontractivity}, the AI model suggested using a version of entropy decay with a defect term. The authors subsequently simplified the estimate, rewrote the proofs, and prepared the manuscript. We take full responsibility for its contents. 

\section{Preliminaries}\label{sec:moments}

For a probability density~$g$ relative to~$\pi$, let $g_x=\Pi_v g$ be the spatial density relative to~$\pi_x$ and $g_{v\mid x}(x,v)=g(x,v)/g_x(x)$ the conditional velocity density relative to~$\varphi_d$ when $g_x(x)>0$. On $\{g_x=0\}$, set $g_{v\mid x}=1$. Following~\cite{LuUnderdampedEntropy}*{Section 2}, write $j=\Pi_v(vg)$ for the velocity current.
Entropy decomposes as
\[
 \Ent_x(g_x)=\int g_x\log g_x\dd\pi_x,\qquad
 \Ent_v(g)=\int g_x\,\Ent_{\varphi_d}(g_{v\mid x})\dd\pi_x,\qquad
 \Ent_\pi(g)=\Ent_x(g_x)+\Ent_v(g).
\]
For a time-dependent density~$g_t$, write $g_{x,t}=\Pi_v g_t$. We suppress the time subscript in calculations at a fixed time.

Since \(g_t\) is the density relative to \(\pi\), it satisfies the forward equation \(\partial_tg=\mathcal L^*g\) in distributions, where \(\mathcal L^*\) is the adjoint of \(\mathcal L\) in \(L^2(\pi)\). Explicitly,
\begin{equation}\label{eq:forward}
 \partial_tg=-v\cdot\nabla_xg
   +(-v\cdot\nabla U)_+\,[g(x,R_xv)-g(x,v)]
   +\gamma(\Pi_vg-g).
\end{equation}
Let \(\nabla_x^*=-\operatorname{div}+\nabla U\cdot\) denote the adjoint of \(\nabla_x\) in \(L^2(\pi_x)\). Taking the zeroth and first velocity moments of~\eqref{eq:forward} gives, in distributions,
\begin{align}
 \partial_t g_x&=\nabla_x^*j,\label{eq:continuity}\\
 \partial_tj&=-\nabla_x g_x
                +\nabla_x^*\!\left(\frac{j\otimes j}{g_x}+\Theta\right)
                -\gamma j+\mathfrak d\, g_x\nabla U\,,\label{eq:moment}\\
 \Theta&=\Pi_v(v\otimes v\,g)-\frac{j\otimes j}{g_x}-g_x\Id,
 \label{eq:stress}\\
 B&=2\int(v\cdot n)_+^2g_{v\mid x}(x,v)\dd\varphi_d(v),\qquad
 \mathfrak d=1-B.
 \label{eq:one-sided-moment}
\end{align}
Here \(\Theta\) is the centered stress tensor, with quotients by \(g_x\) set to zero on \(\{g_x=0\}\). The quantity~$B=B(x)$ is the one-sided second moment of~$g_{v\mid x}$ in the direction~$n=n(x)$, normalized so that $B=1$ when $g_{v\mid x}=1$. The signed force defect~$\mathfrak d=\mathfrak d(x)$ therefore vanishes at local velocity equilibrium. We suppress the spatial arguments of~$B$ and~$\mathfrak d$ in the moment equations.
Equations~\eqref{eq:continuity}--\eqref{eq:moment} parallel the moment equations for underdamped Langevin dynamics; see~\cite{LuUnderdampedEntropy}*{(4.1) and (4.3)}. The only new term is the defect term \(\mathfrak d\,g_x\nabla U\) in~\eqref{eq:moment}, which arises because BPS uses bounces in place of the force \(-\nabla U\).

To derive~\eqref{eq:moment}, multiply~\eqref{eq:forward} by \(v\) and apply \(\Pi_v\). We compute the transport, bounce, and refreshment contributions to \(\partial_tj\) in turn. The transport term  contributes
\[
 -\operatorname{div}\Pi_v(v\otimes v\,g)
 =\nabla_x^*\Pi_v(v\otimes v\,g)-\Pi_v(v\otimes v\,g)\nabla U.
\]
For the bounce term, recall that \(R_x\) is an orthogonal reflection with \(R_x\nabla U=-\nabla U\). The substitution \(v\mapsto R_xv\) therefore preserves \(\varphi_d\) and maps \((-v\cdot\nabla U)_+\) to \((v\cdot\nabla U)_+\), so that
\[
 \Pi_v\!\left[(-v\cdot\nabla U)_+\,v\,g(x,R_xv)\right]
 =\Pi_v\!\left[(v\cdot\nabla U)_+\,R_xv\,g\right].
\]
Hence the bounce contribution is \(\Pi_v[((v\cdot\nabla U)_+R_xv-(-v\cdot\nabla U)_+v)g]\). In the last transport term, write \(\Pi_v(v\otimes v\,g)\nabla U=\Pi_v[(v\cdot\nabla U)v\,g]\) and split \(v\cdot\nabla U=(v\cdot\nabla U)_+-(-v\cdot\nabla U)_+\). The negative parts cancel against the bounce contribution, and the two terms combine into
\[
 \begin{aligned}
 \Pi_v\!\left[(v\cdot\nabla U)_+(R_xv-v)g\right]
 &=-2g_x\nabla U\int(v\cdot n)_+^2g_{v\mid x}(x,v)\dd\varphi_d(v)\\
 &=-g_xB\nabla U.
 \end{aligned}
\]
The first equality uses \(R_xv-v=-2(v\cdot n)n\), \((v\cdot\nabla U)_+=|\nabla U|(v\cdot n)_+\), and \(g=g_xg_{v\mid x}\); the second is the definition~\eqref{eq:one-sided-moment}. Finally, since \(\varphi_d\) has mean zero, refreshment contributes \(\gamma\Pi_v[v(\Pi_vg-g)]=-\gamma j\). Adding the three contributions, and then using the stress decomposition~\eqref{eq:stress} and \(\nabla_x^*(g_x\Id)=-\nabla_x g_x+g_x\nabla U\), gives
\[
 \begin{aligned}
 \partial_tj
 &=\nabla_x^*\Pi_v(v\otimes v\,g)-g_xB\nabla U-\gamma j\\
 &=\nabla_x^*\!\left(\frac{j\otimes j}{g_x}+\Theta+g_x\Id\right)
     -g_xB\nabla U-\gamma j\\
 &=-\nabla_x g_x
     +\nabla_x^*\!\left(\frac{j\otimes j}{g_x}+\Theta\right)
     -\gamma j+\mathfrak d\, g_x\nabla U.
 \end{aligned}
\]
Thus the bounce contribution, combined with the last transport term, gives the force \(-B\,g_x\nabla U\). The defect term \(\mathfrak d\,g_x\nabla U\) is its deviation from the Hamiltonian force \(-g_x\nabla U\).

We next recall the Wasserstein entropy current, as defined in \cite{LuUnderdampedEntropy}.
Let \(T_{g_x}\) be the Brenier map~\cite{Brenier1991} from \(g_x\pi_x\) to \(\pi_x\). The Brenier displacement and Wasserstein entropy current are
\begin{equation}\label{eq:current-definition}
 \xi_{g_x}=x-T_{g_x}(x),\qquad
 \COT(g)=\int j\cdot\xi_{g_x}\dd\pi_x.
\end{equation}
Strong convexity gives \(\pi_x\) a logarithmic Sobolev constant \(m\) by the Bakry--\'Emery criterion~\cite{BakryEmery1985}. The entropy-current bound of Lu~\cite{LuUnderdampedEntropy}*{Lemma 3.2 and its proof} therefore gives
\begin{equation}\label{eq:current-size}
 |\COT(g)|\le\frac2{\sqrt{m}}
       \sqrt{\Ent_x(g_x)\Ent_v(g)}
 \le\frac1{\sqrt{m}}\Ent_\pi(g).
\end{equation}

\section{Entropy and current inequalities}\label{sec:current}
We establish the two differential inequalities that yield entropy decay. The main estimate controls the potential-difference term arising from the defect term. On a truncated region, we use the log-determinant defect defined in~\eqref{eq:logdet-term} together with spatial and conditional velocity entropy. On the complement, the density bound gives an exponentially small remainder.

\begin{proposition}[Entropy inequalities for bounded initial densities]
\label{prop:weak-current}
Suppose \(0\le g_0\le M\) and \(\int g_0\dd\pi=1\). For every constant refreshment rate \(\gamma>0\) and every \(r\ge d+1\), the BPS trajectory satisfies
\begin{align}
 \frac{\dd}{\dd t}\Ent_\pi(g_t)
 &\le-\gamma\Ent_v(g_t),\label{eq:entropy-diss}\\
 \frac{\dd}{\dd t}\COT(g_t)
 &\le-\frac12\Ent_x(g_{x,t})+A_r\Ent_v(g_t)
          -\gamma\COT(g_t)+\eta_r.
 \label{eq:closed-current}
\end{align}
Here
\begin{equation}\label{eq:tail-constant}
 A_r=4+44r,\qquad
 \eta_r=5M\,2^{d/2}(r+3)e^{-r/4}.
\end{equation}
The inequalities hold in distributions in time, and \(\Ent_\pi(g_t)\) and \(\COT(g_t)\) are continuous on \([0,\infty)\).
\end{proposition}

Proposition~\ref{prop:weak-current} parallels the entropy dissipation identity and the current inequality for underdamped Langevin dynamics in~\cite{LuUnderdampedEntropy}*{(3.1) and (6.1)}. Refreshment dissipates the conditional velocity entropy \(\Ent_v\) directly. We therefore keep the current inequality in terms of \(\Ent_v\). The defect term in~\eqref{eq:moment} vanishes at local velocity equilibrium. Without this force correction, the argument of~\cite{LuUnderdampedEntropy} would give
\[
 \frac{\dd}{\dd t}\COT(g_t)
 \le-\Ent_x(g_{x,t})+4\Ent_v(g_t)-\gamma\COT(g_t).
\]
To control the defect term, we use convexity of $U$ and retain the log-determinant defect that the stress bound of~\cite{LuUnderdampedEntropy}*{Lemma 5.3} discards. Comparison with~\eqref{eq:closed-current} shows the price: a coefficient \(A_r\) linear in the cutoff level and the additive remainder \(\eta_r\), which depends on the density bound \(M\).

Subsection~\ref{subsec:current-defects} derives the current inequality for regular densities. Subsection~\ref{subsec:force} controls the defect term and completes the proof of Proposition~\ref{prop:weak-current}. Appendix~\ref{sec:weak} supplies the technical justifications, approximation, and continuity arguments.

\subsection{Current inequality with a log-determinant defect}\label{subsec:current-defects}

Let $J=DT_{g_x}$ denote the almost everywhere derivative of the Brenier map. Define the nonnegative log-determinant defect
\begin{equation}\label{eq:logdet-term}
 D_J(x)=2\log\det\!\left(\frac{\Id+J}{2}\right)-\log\det J.
\end{equation}
The quantity $D_J$ is twice the midpoint Jensen gap of $-\log\det$ between $\Id$ and $J$. Its nonnegativity follows from convexity of $-\log\det$.

Let \(\Sigma(x)\) be the covariance under \(g_{v\mid x}\), so that the centered stress tensor~\eqref{eq:stress} is \(\Theta=g_x(\Sigma-\Id)\). Let \(k(x)=\Ent_{\varphi_d}(g_{v\mid x})\) be the relative entropy of \(g_{v\mid x}\), so that \(\Ent_v(g)=\int g_xk\dd\pi_x\). Here and below, $U-U(T_{g_x})$ means $U(x)-U(T_{g_x}(x))$.

\begin{proposition}[Entropy-current inequality]\label{prop:current}
Suppose \(U\) is smooth and \(g_t\) is a curve of probability densities with \(0\le g_t\le M\) satisfying~\eqref{eq:continuity}--\eqref{eq:one-sided-moment}. Assume that \((g_{x,t},j_t)\) is a regular density-current pair in the sense of Li and Lu~\cite{LiLuHypercontractivity}*{Definition 2.5}, with spatial reference \(\pi_x\). Assume also that \(\Theta_t\) is smooth in position and continuously differentiable in time, and that \(\nabla g_{x,t},\nabla_x^*\Theta_t\in L^2(g_{x,t}^{-1}\pi_x)\), locally uniformly in time. Then, in distributions in time,
\begin{equation}\label{eq:current-full}
 \begin{aligned}
 \frac{\dd}{\dd t}\COT(g_t)
 &\le-\Ent_x(g_{x,t})+4\Ent_v(g_t)-\gamma\COT(g_t)\\
 &\quad+\int g_{x,t}\bigl\{\mathfrak d[U-U(T_{g_x})]-D_J\bigr\}\dd\pi_x\,.
 \end{aligned}
\end{equation}
\end{proposition}

\begin{proof}
We bound the two terms in the weak Wasserstein acceleration inequality of Li and Lu~\cite{LiLuHypercontractivity}*{Lemma 2.7}:
\[
 \frac{\dd}{\dd t}\COT(g_t)
 \le\int\frac{|j|^2}{g_x}\dd\pi_x
 +\int\xi_{g_x}\cdot\left[
    \partial_tj-\nabla_x^*\!\left(\frac{j\otimes j}{g_x}\right)
   \right]\dd\pi_x.
\]
Substituting~\eqref{eq:moment} and integrating by parts as justified in Lemma~\ref{lem:spatial-bv} gives
\begin{equation}\label{eq:spatial-bound}
 \begin{aligned}
 \frac{\dd}{\dd t}\COT(g_t)
 &\le-\gamma\COT(g)\\
 &\quad+\int g_x\bigl[|j/g_x|^2+\Sigma:(\Id-J)-B\nabla U\cdot\xi_{g_x}\bigr]\dd\pi_x,
 \end{aligned}
\end{equation}
where the singular contribution is nonpositive and the colon denotes the Frobenius inner product.

The centered conditional velocity law has covariance $\Sigma$ and relative entropy $k-|j/g_x|^2/2$. Apply entropy duality to this law with test function $v\cdot(\Id-J)v/4$. Since $J\succeq0$, its Gaussian exponential moment is $\det((\Id+J)/2)^{-1/2}$, as in~\cite{LuUnderdampedEntropy}*{Lemma 5.1}. Hence
\begin{equation*}
 \begin{aligned}
 |j/g_x|^2+\Sigma:(\Id-J)
 &\le4k-|j/g_x|^2-2\log\det\!\left(\frac{\Id+J}{2}\right)\\
 &\le4k-2\log\det\!\left(\frac{\Id+J}{2}\right).
 \end{aligned}
\end{equation*}
Convexity gives $\nabla U\cdot\xi_{g_x}\ge U-U(T_{g_x})$. Since $B\ge0$, the integrand in~\eqref{eq:spatial-bound} is therefore at most
\begin{equation*}
 4k-B[U-U(T_{g_x})]-2\log\det\!\left(\frac{\Id+J}{2}\right).
\end{equation*}
The Monge--Amp\`ere identity~\cite{McCann1997}
\begin{equation}\label{eq:MA}
 \log g_x=U(x)-U(T_{g_x}(x))+\log\det J
\end{equation}
holds $g_x\pi_x$-almost everywhere, with $\det J>0$ there; see Lemma~\ref{lem:jacobian-logs}. By~\eqref{eq:logdet-term} and~\eqref{eq:MA}, the last display equals
\begin{equation*}
 4k-\log g_x+\mathfrak d[U-U(T_{g_x})]-D_J.
\end{equation*}
Integrating against $g_x\pi_x$ in~\eqref{eq:spatial-bound} proves~\eqref{eq:current-full}.
\end{proof}

\subsection{Control of the \texorpdfstring{defect term}{defect term}}\label{subsec:force}
We bound the last integral in~\eqref{eq:current-full} using $D_J$, up to entropy terms and an exponentially small remainder. The coefficient of conditional velocity entropy is linear in the cutoff level and independent of the condition number.

We estimate the last integral separately on $H_r$ and its complement, where
\begin{equation}\label{eq:good-set}
 H_r=\{x:|U(x)-U(T_{g_x}(x))|\le r\}\,.
\end{equation}
The choice of~$r$ balances a velocity-entropy coefficient linear in~$r$ against an exponentially small remainder.

\begin{lemma}[A one-sided Gaussian exponential moment]\label{lem:gaussian-moment}
For every unit vector~$n$ and $z<1/4$,
\begin{equation}\label{eq:bounce-mgf}
 \begin{aligned}
 \psi(z)&:=\log\int e^{z[2(v\cdot n)_+^2-1]}\dd\varphi_d(v)\\
 &=-z+\log\frac{1+(1-4z)^{-1/2}}2
 \le\frac{5z^2}{2(1-4z_+)}\,.
 \end{aligned}
\end{equation}
\end{lemma}
\begin{proof}
The formula follows by Gaussian symmetry and direct integration. Direct differentiation gives
\begin{equation*}
 \psi''(z)\le\begin{cases}
 5,&z\le0,\\
 5(1-4z)^{-3},&0\le z<1/4.
 \end{cases}
\end{equation*}
Integrating twice from $\psi(0)=\psi'(0)=0$ proves~\eqref{eq:bounce-mgf}.
\end{proof}

\begin{lemma}[Defect-term estimate on $H_r$]\label{lem:central}
Let $g$ be a probability density relative to~$\pi$, with $0\le g\le M$. For $r\ge d+1$ and $g_x\pi_x$-almost every point in~$H_r$,
\begin{equation}\label{eq:central-pointwise}
 \begin{aligned}
 &\mathfrak d[U-U(T_{g_x})]-D_J\\
 &\qquad\le44r k
 +\frac12\left(\log g_x-1+\frac1{g_x}\right)-\frac12D_J\,.
 \end{aligned}
\end{equation}
Consequently, with the integral understood in the extended sense,
\begin{equation}\label{eq:central}
 \int_{H_r}g_x\bigl\{\mathfrak d[U-U(T_{g_x})]-D_J\bigr\}\dd\pi_x
 \le44r\Ent_v(g)+\frac12\Ent_x(g_x)\,.
\end{equation}
\end{lemma}
\begin{proof}
Fix a point in~$H_r$ where $g_x>0$, $J\succ0$, and~\eqref{eq:MA} holds. Apply entropy duality to the probability measures $g_{v\mid x}\varphi_d$ and~$\varphi_d$ with test function
\begin{equation*}
 \frac{U-U(T_{g_x})}{44r}\bigl[1-2(v\cdot n)_+^2\bigr]\,.
\end{equation*}
In the Gaussian exponential-moment bound~\eqref{eq:bounce-mgf}, take
\begin{equation*}
 z=-\frac{U-U(T_{g_x})}{44r}\,.
\end{equation*}
Since $|U-U(T_{g_x})|\le r$, we have $1-4z_+\ge10/11$. Therefore
\begin{equation*}
 \begin{aligned}
 \frac{\mathfrak d[U-U(T_{g_x})]}{44r}
 &\le k+\log\int e^{z[2(v\cdot n)_+^2-1]}\dd\varphi_d(v)\\
 &\le k+\frac{[U-U(T_{g_x})]^2}{704r^2}\,.
 \end{aligned}
\end{equation*}
Multiplying by~$44r$ gives
\begin{equation}\label{eq:force-duality}
 \mathfrak d[U-U(T_{g_x})]
 \le44r k+\frac{[U-U(T_{g_x})]^2}{16r}\,.
\end{equation}

We next bound the square using~$D_J$. For $|z|\le1/2$, we have $\log\cosh z\ge z^2/4$. Thus
\begin{equation*}
 \frac{[U-U(T_{g_x})]^2}{16r}
 \le r\log\cosh\!\left(\frac{U-U(T_{g_x})}{2r}\right)\,.
\end{equation*}
Let $\lambda_1,\ldots,\lambda_d$ be the eigenvalues of~$J$. The Monge--Amp\`ere identity and the definition of~$D_J$ give
\begin{align*}
 U-U(T_{g_x})&=\log g_x-\sum_{i=1}^d\log\lambda_i,\\
 \frac12D_J&=\sum_{i=1}^d\log\cosh\!\left(\frac12\log\lambda_i\right)\,.
\end{align*}
Apply Jensen's inequality to the convex function $\log\cosh$ with weight $1/r$ at each of $\frac12\log g_x,-\frac12\log\lambda_1,\ldots,-\frac12\log\lambda_d$, and weight $1-(d+1)/r$ at zero. The weights are nonnegative because $r\ge d+1$. Evenness of $\log\cosh$ then gives
\begin{equation}\label{eq:force-absorption}
 \begin{aligned}
 \frac{[U-U(T_{g_x})]^2}{16r}
 &\le r\log\cosh\!\left(\frac{U-U(T_{g_x})}{2r}\right)\\
 &\le\log\cosh\!\left(\frac12\log g_x\right)+\frac12D_J\,.
 \end{aligned}
\end{equation}

To control the remaining spatial term, use $\log u\le u-1$ for $u>0$:
\begin{equation*}
 \begin{aligned}
 \log\cosh\!\left(\frac12\log g_x\right)
 &=\frac12\log g_x+\log\frac{1+1/g_x}{2}\\
 &\le\frac12\left(\log g_x-1+\frac1{g_x}\right)\,.
 \end{aligned}
\end{equation*}
Substitute this inequality and~\eqref{eq:force-absorption} into~\eqref{eq:force-duality}, and subtract~$D_J$. This proves~\eqref{eq:central-pointwise}.

For the integrated estimate, the nonnegative spatial entropy density satisfies
\begin{equation*}
 g_x\log g_x-g_x+1\ge0,\qquad
 \int(g_x\log g_x-g_x+1)\dd\pi_x=\Ent_x(g_x)\,.
\end{equation*}
At points where $g_x>0$, this density equals $g_x(\log g_x-1+1/g_x)$. The set $\{g_x=0\}$ has zero $g_x\pi_x$-mass. Consequently,
\begin{equation*}
 \int_{H_r}g_x\left(\log g_x-1+\frac1{g_x}\right)\dd\pi_x
 \le\Ent_x(g_x)\,.
\end{equation*}
On~$H_r$, the defect term is integrable and $D_J\ge0$, so the integral is well defined. Discarding~$-D_J/2$ in~\eqref{eq:central-pointwise} and using $\int_{H_r}g_xk\dd\pi_x\le\Ent_v(g)$ proves~\eqref{eq:central}.
\end{proof}

It remains to bound the defect term outside~$H_r$. We first record tail and moment bounds for the potential.

\begin{lemma}[Tail and truncated second-moment bounds for the potential]\label{lem:potential}
Let $U$ be a convex potential with minimum zero and finite partition function. Then, for every $s\ge0$,
\begin{align*}
 \pi_x(U>s)&\le2^d e^{-s/2},\\
 \int U^2\mathbf1_{\{U>s\}}\dd\pi_x
 &\le2^d(s^2+4s+8)e^{-s/2}\,.
\end{align*}
\end{lemma}
\begin{proof}
Translate a minimizer to the origin, so that $U(0)=0$. Convexity then gives $U(2y)\ge2U(y)$, and the change of variables $x=2y$ yields
\begin{equation*}
 \int e^{U/2}\dd\pi_x
 =\frac1Z\int e^{-U(x)/2}\dd x
 =\frac{2^d}{Z}\int e^{-U(2y)/2}\dd y
 \le2^d\,.
\end{equation*}
Markov's inequality gives the tail bound. Integrating it gives
 \begin{align*}
 \int U^2\mathbf1_{\{U>s\}}\dd\pi_x
 &=s^2\pi_x(U>s)+\int_s^\infty2t\pi_x(U>t)\dd t\\
 &\le2^d(s^2+4s+8)e^{-s/2}\,. \qedhere
 \end{align*}
\end{proof}

These bounds make the defect term outside~$H_r$ exponentially small in the cutoff level.

\begin{lemma}[Defect-term estimate on $H_r^c$]\label{lem:tail}
Let $g$ be a probability density relative to~$\pi$, with $0\le g\le M$. For every $r\ge d+1$,
\begin{equation}\label{eq:tail}
 \int_{H_r^c}g_x|U-U(T_{g_x})|\,|\mathfrak d|\dd\pi_x
 \le5M2^{d/2}(r+3)e^{-r/4}=\eta_r\,.
\end{equation}
\end{lemma}
\begin{proof}
Since $U\ge0$, the events $U-U(T_{g_x})>r$ and $U-U(T_{g_x})<-r$ give, respectively,
\begin{equation*}
 [U-U(T_{g_x})]^2\le U^2\mathbf1_{\{U>r\}},\qquad
 [U-U(T_{g_x})]^2\le U(T_{g_x})^2\mathbf1_{\{U(T_{g_x})>r\}}\,.
\end{equation*}
Hence $g_x\le M$ and $(T_{g_x})_\#(g_x\pi_x)=\pi_x$ imply
\begin{equation*}
 \int_{H_r^c}g_x[U-U(T_{g_x})]^2\dd\pi_x
 \le(M+1)\int U^2\mathbf1_{\{U>r\}}\dd\pi_x\,.
\end{equation*}
Conditional Jensen and $g\le M$ give
\begin{equation*}
 \int g_xB\dd\pi_x\le M,\qquad
 \int g_xB^2\dd\pi_x\le6M\,.
\end{equation*}
Since $M\ge1$, it follows that $\int g_x(1+B)^2\dd\pi_x\le9M$. Cauchy--Schwarz therefore gives
\begin{equation*}
 \begin{aligned}
 &\int_{H_r^c}g_x|U-U(T_{g_x})|\,|\mathfrak d|\dd\pi_x\\
 &\quad\le3\sqrt{M(M+1)\int U^2\mathbf1_{\{U>r\}}\dd\pi_x}\,.
 \end{aligned}
\end{equation*}
Use Lemma~\ref{lem:potential}, $M+1\le2M$, and $\sqrt{r^2+4r+8}\le r+3$ to obtain~\eqref{eq:tail}.
\end{proof}

\begin{proof}[Proof of Proposition~\ref{prop:weak-current}]
Entropy dissipation requires no smoothing. The semigroup of transport and bouncing preserves $\pi$ and contracts relative entropy. Conditioning on the last refreshment in an interval of length $h>0$ and using convexity gives
\begin{equation*}
 \Ent_\pi(g_{t+h})
 \le e^{-\gamma h}\Ent_\pi(g_t)
   +\gamma\int_0^h e^{-\gamma(h-u)}\Ent_x(g_{x,t+u})\dd u.
\end{equation*}
Taking difference quotients in distributions proves~\eqref{eq:entropy-diss}.

For the current inequality, first assume the regularity hypotheses of Proposition~\ref{prop:current}. Combining~\eqref{eq:current-full} with Lemmas~\ref{lem:central} and~\ref{lem:tail}, and discarding the nonpositive tail contribution of~$D_J$, gives
\begin{equation*}
 \frac{\dd}{\dd t}\COT(g_t)
 \le-\frac12\Ent_x(g_{x,t})+(4+44r)\Ent_v(g_t)-\gamma\COT(g_t)+\eta_r\,.
\end{equation*}
This is~\eqref{eq:closed-current} with~\eqref{eq:tail-constant}. Lemma~\ref{lem:approximation} extends the current inequality to bounded initial densities and gives continuity of the entropy and current.
\end{proof}

\section{Entropy decay and query complexity}\label{sec:closure}
We prove the entropy-decay estimate and the resulting convergence-time and gradient-query bounds from a Gaussian cold start.

\begin{proof}[Proof of Theorem~\ref{thm:master}]
Fix \(\gamma>0\) and \(0<\delta\le1/2\). We first work with an arbitrary cutoff \(r\ge d+1\). For \(\epsilon>0\), define the modified entropy
\[
 \mathcal H_\epsilon(g)=\Ent_\pi(g)+\epsilon\COT(g).
\]
Equation~\eqref{eq:current-size} gives
\begin{equation}\label{eq:equiv}
 (1-\epsilon/\sqrt{m})\Ent_\pi(g)
 \le\mathcal H_\epsilon(g)\le
 (1+\epsilon/\sqrt{m})\Ent_\pi(g).
\end{equation}
Young's inequality and~\eqref{eq:current-size} give
\[
 \gamma|\COT(g)|
 \le\frac14\Ent_x(g_x)+\frac{4\gamma^2}{m}\Ent_v(g).
\]
Combining this bound with the two inequalities in Proposition~\ref{prop:weak-current} gives
\begin{equation*}
 \frac{\dd}{\dd t}\mathcal H_\epsilon(g_t)
 \le-\frac\epsilon4\Ent_x(g_{x,t})
   -\bigl[\gamma-\epsilon(A_r+4\gamma^2/m)\bigr]\Ent_v(g_t)
   +\epsilon\eta_r.
\end{equation*}
Choose
\[
 \epsilon=\frac{\gamma}{2(A_r+4\gamma^2/m)},
\]
so that the coefficient in brackets is \(\gamma/2\). This choice also satisfies \(\epsilon/\sqrt{m}\le1/(8\sqrt{A_r})\le1/8\). Since \(\gamma/2\ge\epsilon/4\), the differential inequality and~\eqref{eq:equiv} yield
\[
 \frac{\dd}{\dd t}\mathcal H_\epsilon(g_t)
 \le-\frac\epsilon4\Ent_\pi(g_t)+\epsilon\eta_r
 \le-\frac{2\epsilon}{9}\mathcal H_\epsilon(g_t)+\epsilon\eta_r.
\]
The continuity in Proposition~\ref{prop:weak-current} lets Gr\"onwall's inequality start at \(t=0\). Together with the two sides of~\eqref{eq:equiv} and the value of \(\epsilon\), it gives
\begin{equation}\label{eq:cutoff-decay}
 \Ent_\pi(g_t)
 \le2\exp\!\left[-\frac{\gamma t}{9(A_r+4\gamma^2/m)}\right]
       \Ent_\pi(g_0)+6\eta_r.
\end{equation}

To obtain the prescribed additive error, choose
\[
 r=K\bigl[d+\log M+\log(1/\delta)\bigr],
\]
where \(K\ge2\) is a sufficiently large universal constant. Since \(d\ge1\) and \(M\ge1\), this choice satisfies \(r\ge d+1\). Since \(\log(r+3)\le r/8+C\), the definition~\eqref{eq:tail-constant} gives \(6\eta_r\le\delta\) when \(K\) is large enough. Moreover,
\[
 A_r\le C\bigl[d+\log M+\log(1/\delta)\bigr].
\] Substituting these bounds into~\eqref{eq:cutoff-decay} proves~\eqref{eq:master}.
\end{proof}

\begin{proof}[Proof of Corollary~\ref{cor:cold}]
We apply Theorem~\ref{thm:master} with \(M=\kappa^{d/2}\). The initial density satisfies
\[
 g_0(x,v)=Z\left(\frac L{2\pi}\right)^{d/2}
              e^{U(x)-L|x|^2/2}\le\kappa^{d/2}=M.
\]
The inequality uses \(U(x)\le L|x|^2/2\) and \(Z\le(2\pi/m)^{d/2}\). Hence \(\Ent_\pi(g_0)\le\log M\).

Let \(c\) be the constant in Theorem~\ref{thm:master}, reduced if necessary so that \(c\le1\). Set
\begin{equation}\label{eq:ell-choice}
 \ell=d+\log M+\log(2/\delta).
\end{equation}
We apply the theorem with additive error \(\delta/2\) and choose
\begin{align}
 \gamma&=\sqrt{m\ell},\label{eq:gamma-choice}\\
 T&=\frac2c\sqrt{\frac{\ell}{m}}
       \log\!\left(\frac{4(1+\log M)}{\delta}\right).
       \label{eq:T-choice}
\end{align}
The refreshment rate balances \(\ell\) and \(\gamma^2/m\) in the denominator of the decay rate. With additive error \(\delta/2\), the rate in~\eqref{eq:master} is \(c\sqrt{m}/(2\sqrt\ell)\). The horizon in~\eqref{eq:T-choice} makes the exponentially decaying term at most \(\delta/2\), so \(\KL(\mu_0\mathcal P_T\|\pi)\le\delta\). Since \(\log M=(d/2)\log\kappa\) and \(d\ge1\),
\[
 \ell\le C\bigl[d\log(e\kappa)+\log(1/\delta)\bigr].
\]
Substituting this bound into~\eqref{eq:T-choice} gives~\eqref{eq:time-order}.
\end{proof}

\begin{proof}[Proof of Corollary~\ref{cor:queries}]
Choose the refreshment rate and horizon from the proof of Corollary~\ref{cor:cold}, and use window length \((Ld)^{-1/2}\). We verify the horizon restrictions in the query estimate of Lu and Luo~\cite{LuLuo}*{Proposition 9, (27)}, which reads
\[
 Q_\nabla(T)\le5\sqrt{Ld}\,T,
 \qquad T\ge\max\{L^{-1/2},\gamma/(4L)\}.
\]
The Gaussian cold start and constant refreshment rate satisfy the remaining hypotheses of that estimate.

The horizon in~\eqref{eq:T-choice} satisfies the first restriction because \(c\le1\), \(\ell\ge1\), \(\kappa\ge1\), and \(\log(4(1+\log M)/\delta)\ge\log8\). For the second restriction,~\eqref{eq:gamma-choice} gives
\[
 \frac{T}{\gamma/(4L)}
 =\frac{8\kappa}{c}\log\!\left(\frac{4(1+\log M)}{\delta}\right)>1.
\]
Exactness of windowed thinning and Corollary~\ref{cor:cold} give the entropy guarantee. Substituting~\eqref{eq:time-order} into the query bound proves~\eqref{eq:query}.
\end{proof}

\appendix
\section{Approximation for bounded initial densities}\label{sec:weak}
We justify the logarithmic terms and integration by parts used in Proposition~\ref{prop:current}, and the approximation and continuity statements used in Proposition~\ref{prop:weak-current}. The approximation follows~\cite{LuUnderdampedEntropy}*{Section 7}, with the regularity conditions of~\cite{LiLuHypercontractivity}*{Definition 2.5} verified below.

\begin{lemma}[Jacobian logarithms]\label{lem:jacobian-logs}
Let $g_x$ be a probability density relative to $\pi_x$ with $0\le g_x\le M$. The Alexandrov derivative $J=DT_{g_x}$ is positive definite $g_x\pi_x$-almost everywhere, and~\eqref{eq:MA} holds there. Moreover, $\log\det J\in L^1(g_x\pi_x)$.
\end{lemma}

\begin{proof}
Both transport marginals are absolutely continuous. The monotone change-of-variables theorem~\cite{McCann1997}*{Theorem 4.4 and Remark 4.5} gives
\begin{equation*}
 Z^{-1}e^{-U(T_{g_x}(x))}\det J(x)
 =g_x(x)Z^{-1}e^{-U(x)}\,.
\end{equation*}
Since $g_x>0$ $g_x\pi_x$-almost everywhere, this identity gives $\det J>0$. Together with $J\succeq0$, it proves $J\succ0$ and~\eqref{eq:MA} there. The bound $g_x\le M$ gives $\int g_x|\log g_x|\dd\pi_x\le\log M+e^{-1}$. Both $U$ and $U(T_{g_x})$ are integrable against~$g_x\pi_x$, by the density bound, Gaussian tails, and the pushforward identity. Equation~\eqref{eq:MA} therefore gives the claimed logarithmic integrability.
\end{proof}

\begin{lemma}[Spatial integration by parts]\label{lem:spatial-bv}
Under the hypotheses of Proposition~\ref{prop:current}, the integration by parts in~\eqref{eq:spatial-bound} is valid, and its singular contribution is nonpositive. Moreover, $D_J\in L^1(g_x\pi_x)$, locally uniformly in time.
\end{lemma}

\begin{proof}
We follow the localized BV argument in~\cite{LuUnderdampedEntropy}*{Lemma 5.3}, retaining the exact logarithmic determinant. The regular density-current assumption gives a positive spatial density with full support. The Brenier potential is therefore finite throughout \(\mathbb R^d\), and \(T_{g_x}\in BV_{\rm loc}(\mathbb R^d)\). Its distributional derivative is \(J\dd x+N\dd\sigma\), where \(\sigma\) is a nonnegative singular measure and \(N\succeq0\). By the regularity assumptions, \(\Sigma\) has a continuous positive semidefinite representative. Under a nonnegative cutoff \(\chi_R\), the singular part is therefore
\[
 -\int\chi_R g_x\,Z^{-1}e^{-U}\Sigma:N\dd\sigma\le0.
\]
The cutoff boundary term vanishes as \(R\to\infty\) by Cauchy--Schwarz, since both \(\int g_x\|\Sigma\|^2\dd\pi_x\le\int|v|^4g\dd\pi\) and \(\int g_x|\xi_{g_x}|^2\dd\pi_x\) are finite. The weighted $L^2$ assumptions justify the remaining pairings. To justify the global BV decomposition, set $K=Z^{-1}e^{-U}g_x\Sigma$. Since
\begin{equation*}
 \operatorname{div}K
 =Z^{-1}e^{-U}(-\nabla_x^*\Theta+\nabla g_x-g_x\nabla U),
\end{equation*}
these assumptions and the second spatial moment give $\int|\operatorname{div}K|^2/(Z^{-1}e^{-U}g_x)\dd x<\infty$. For increasing cutoffs $\chi_R$, the nonnegative quantity
\begin{equation*}
 S_R=\int\chi_R g_x\Sigma:J\dd\pi_x
       +\int\chi_R Z^{-1}e^{-U}g_x\Sigma:N\dd\sigma
\end{equation*}
therefore has a finite limit: the localized integration by parts expresses it as $\int\chi_R g_x\tr\Sigma\dd\pi_x-\int\chi_R\xi_{g_x}\cdot(-\operatorname{div}K)\dd x$ plus the vanishing boundary term. Thus $\int g_x\Sigma:J\dd\pi_x<\infty$, and no divergent trace terms are subtracted in the global decomposition. The force pairing converges by Cauchy--Schwarz, which justifies the integration by parts in~\eqref{eq:spatial-bound}.

For the logarithmic integrals, Gaussian entropy duality for the centered conditional velocity law, with test function $v\cdot(\Id-J)v/4$, gives the pointwise bounds
\begin{equation*}
 -d\log2
 \le\log\det\!\left(\frac{\Id+J}{2}\right)
 \le2k+\frac12\Sigma:J.
\end{equation*}
The lower bound uses $J\succeq0$. The upper bound and the integrability of $k$ and $\Sigma:J$ prove that this logarithm is integrable. Lemma~\ref{lem:jacobian-logs} then gives $D_J\in L^1(g_x\pi_x)$, locally uniformly in time by the weighted $L^2$ assumptions and $g_t\le M$.
\end{proof}

\begin{lemma}[Extension to bounded initial densities]\label{lem:approximation}
Under the hypotheses of Proposition~\ref{prop:weak-current}, the current inequality~\eqref{eq:closed-current} holds in distributions in time. Both $\Ent_\pi(g_t)$ and $\COT(g_t)$ are continuous on $[0,\infty)$.
\end{lemma}

\begin{proof}
\emph{Regularization.} Assume first that $g_0\ge\lambda>0$, so $\lambda\le g_t\le M$ for all~$t$. We convolve the trajectory in position at every time. Let $\zeta_s=N(0,s\Id)$, let $U_s$ be the potential of $\pi_x^{(s)}=\pi_x*\zeta_s$, normalized to have minimum zero, and let $g_t^{(s)}$ be the density of $\mu_t*_x\zeta_s$ relative to $\pi_x^{(s)}\otimes\varphi_d$. Averaging gives $\lambda\le g_t^{(s)}\le M$. The potential $U_s$ is smooth and
\begin{equation}\label{eq:regularized-curvature}
 m_s\Id\preceq\nabla^2U_s\preceq L\Id,
 \qquad m_s=\frac{m}{1+ms}\,.
\end{equation}
Indeed, let $X\sim\pi_x$ and $Z\sim N(0,\Id)$ be independent, and set $Y=X+\sqrt{s}Z$. Differentiating the convolution gives
\begin{align*}
 \nabla^2U_s(y)&=s^{-1}\Id-s^{-2}\operatorname{Cov}(X\mid Y=y)\\
 &=\E[\nabla^2U(X)\mid Y=y]-\operatorname{Cov}(\nabla U(X)\mid Y=y)\,.
\end{align*}
The conditional potential has Hessian at least $(m+s^{-1})\Id$, so the Brascamp--Lieb inequality~\cite{BrascampLieb1976} bounds its covariance by $(m+s^{-1})^{-1}\Id$. The two identities give the lower and upper bounds in~\eqref{eq:regularized-curvature}. Conditional Jensen and $\nabla U_s(Y)=\E[\nabla U(X)\mid Y]$ also give
\begin{equation}\label{eq:regularized-force-moment}
 \int|\nabla U_s|^2\dd\pi_x^{(s)}
 \le\int|\nabla U|^2\dd\pi_x
 =\int\Delta U\dd\pi_x\le Ld\,.
\end{equation}

We verify the regularity hypotheses of Proposition~\ref{prop:current} for this convolved path. Write $p_s$ and $f_t$ for the Lebesgue densities of~$\pi_x^{(s)}$ and~$\mu_t$, and set $f_{t,s}=f_t*_x\zeta_s$. Define
\begin{equation*}
 \rho_s=\int f_{t,s}\dd v=p_sg_{x,t}^{(s)},\qquad
 \jmath_s=\int vf_{t,s}\dd v=p_sj_t^{(s)},\qquad
 u_s=\frac{\jmath_s}{\rho_s}=\frac{j_t^{(s)}}{g_{x,t}^{(s)}}\,.
\end{equation*}
For the regularity verification, $C_{\lambda,s}$ denotes finite constants depending also on $d,m,L,M,\gamma$, but not on~$t$.

For a stationary BPS process~$\mathsf Z_t$, stochastic continuity gives $\|\mathcal P_t^*h-h\|_{L^1(\pi)}\le\E|h(\mathsf Z_0)-h(\mathsf Z_t)|\to0$ for bounded continuous~$h$. Density and contraction extend this to $L^1(\pi)$. Since $g_t\le M$, truncation gives, for every nonnegative polynomial~$P$,
\begin{equation}\label{eq:regularized-weighted-continuity}
 \int P(x,v)|g_t-g_u|\dd\pi\longrightarrow0
 \qquad(t\to u)\,.
\end{equation}
For $\ell=0,1,2$ and every spatial multi-index~$\beta$, testing the weak BPS equation against $D_x^\beta\zeta_s(x-y)v^{\otimes\ell}$ gives continuous time derivatives of the convolved moments. The tests and their generator images have polynomial growth, locally uniformly in~$x$, so~\eqref{eq:regularized-weighted-continuity} justifies this in every local spatial derivative norm. Hence $g_{x,t}^{(s)},j_t^{(s)},\Theta_t^{(s)}\in C^1_tC^\infty_x$.

The conditional velocity density is bounded by $M/\lambda$, so $|u_s|^2\le M/\lambda$. The spatial Gaussian convolution estimate in~\cite{LiLuHypercontractivity}*{proof of Lemma B.15} gives
\begin{equation}\label{eq:regularized-fisher}
 \mathcal F_{4,s}:=\int(1+|v|^4)
       \frac{|\nabla_x f_{t,s}|^2}{f_{t,s}}\dd x\dd v
 \le\frac ds\int(1+|v|^4)f_t\dd x\dd v
 \le C_{\lambda,s}\,.
\end{equation}
Let $A_s=\rho_s\Sigma_s$, where $\Sigma_s$ is the conditional covariance, and set $H_s=A_s-\rho_s\Id=p_s\Theta_t^{(s)}$. The centered-moment differentiation and Cauchy--Schwarz estimate in~\cite{LiLuHypercontractivity}*{proof of Proposition B.18} use only the bound on $u_s$ and~\eqref{eq:regularized-fisher}; they require no time or velocity convolution. With $\nabla_s^*=-\operatorname{div}+\nabla U_s\cdot$, these estimates and~\eqref{eq:regularized-force-moment} give
\begin{align}
 \int\frac{|\nabla g_{x,t}^{(s)}|^2}{g_{x,t}^{(s)}}\dd\pi_x^{(s)}
 &\le C_{\lambda,s},\label{eq:regularized-density-derivative}\\
 \int\frac{|\nabla_s^*\Theta_t^{(s)}|^2}{g_{x,t}^{(s)}}\dd\pi_x^{(s)}
 &=\int\frac{|\operatorname{div}H_s|^2}{\rho_s}\dd x
 \le C_{\lambda,s},\label{eq:regularized-stress-derivative}\\
 \int\frac{|\operatorname{div}A_s|^2}{\rho_s}\dd x
 &\le C_{\lambda,s},\label{eq:regularized-covariance-derivative}
\end{align}

For the BPS acceleration, define
\begin{equation*}
 F_s^{\rm b}
 =\left[\int2(v\cdot n)_+^2\nabla U\,f_t(\cdot,v)\dd v\right]*\zeta_s\,.
\end{equation*}
The absolute moment equations give
\begin{equation*}
 \partial_t\rho_s+\operatorname{div}\jmath_s=0,\qquad
 \rho_sa_s=-\operatorname{div}A_s-\gamma\jmath_s-F_s^{\rm b},
 \qquad a_s=\partial_tu_s+(u_s\cdot\nabla)u_s\,.
\end{equation*}
Conditional Cauchy--Schwarz in the joint convolution and $\int[2(v\cdot n)_+^2]^2\dd\varphi_d=6$ imply $\int|F_s^{\rm b}|^2/\rho_s\le6MLd$. Therefore
\begin{equation}\label{eq:regularized-acceleration}
 \int|a_s|^2g_{x,t}^{(s)}\dd\pi_x^{(s)}
 \le C_{\lambda,s}\,.
\end{equation}
These bounds, normalization, the density bounds, and the continuity equation verify~\cite{LiLuHypercontractivity}*{Definition 2.5} and the additional assumptions of Proposition~\ref{prop:current}. No uniformity in $\lambda,s\downarrow0$ is needed here.

\emph{Moment equations of the regularized density.} Position convolution commutes with transport and refreshment but changes the defect term. At a fixed time, the conditional law of $x$ given $x+\sqrt{s}Z=y$, for $(x,v)\sim\mu_t$ and independent $Z\sim N(0,\Id)$, is
\[
 w_y(\dd x)=\frac{g_x(x)\zeta_s(y-x)\pi_x(\dd x)}{\int g_x(x')\zeta_s(y-x')\pi_x(\dd x')}.
\]
The conditional velocity density is $g_{v\mid x}^{(s)}(y,v)=\int g_{v\mid x}(x,v)\,w_y(\dd x)$. In the convolved moment equations, the defect term becomes $\widetilde{\mathfrak d}_s\,g_x^{(s)}\nabla U_s-e_s$, where
\begin{align*}
 \widetilde B_s(y)&=\int B\dd w_y,\qquad
 \widetilde{\mathfrak d}_s(y)=1-\widetilde B_s(y),\\
 e_s(y)&=g_x^{(s)}(y)\int B(x)\,[\nabla U(x)-\nabla U_s(y)]\,w_y(\dd x)\,.
\end{align*}

\emph{Current inequality for the regularized density.} The proof of Proposition~\ref{prop:current} gives~\eqref{eq:current-full} for the convolved moments with $U_s$ and $\widetilde B_s\ge0$, and the additional term $-\int\xi_{g_x^{(s)}}\cdot e_s\dd\pi_x^{(s)}$. To control the averaged defect term, put $\widetilde k_s(y)=\int k\dd w_y$. The Gaussian bound~\eqref{eq:bounce-mgf} is independent of~$n(x)$, so averaging entropy duality at each fixed scalar $z<1/4$ gives
\begin{equation}\label{eq:averaged-duality}
 -z\widetilde{\mathfrak d}_s(y)\le\widetilde k_s(y)+\psi(z)\,.
\end{equation}
Also, conditional Jensen's inequality gives
\begin{equation*}
 \int g_x^{(s)}\widetilde B_s\dd\pi_x^{(s)}\le M,\qquad
 \int g_x^{(s)}\widetilde B_s^2\dd\pi_x^{(s)}\le6M\,.
\end{equation*}
These are the velocity inputs in the proofs of Lemmas~\ref{lem:central} and~\ref{lem:tail}, which therefore apply with $U_s$, $\widetilde B_s$, $\widetilde{\mathfrak d}_s$, and $\widetilde k_s$, with the same remainder~$\eta_r$. Moreover, $\int g_x^{(s)}\widetilde k_s\dd\pi_x^{(s)}=\Ent_v(g)$, while convexity gives $\Ent_v(g^{(s)})\le\Ent_v(g)$. Thus the entropy cost is controlled by that of the original trajectory.

For the commutator, let $X\sim\pi_x$, $V\sim\varphi_d$, and $Z\sim N(0,\Id)$ be independent, and put $Y=X+\sqrt s\,Z$. Set $H=2(V\cdot n(X))_+^2$ and $D=\nabla U(X)-\nabla U_s(Y)$. Then
\begin{equation*}
 e_s(Y)=\E[g(X,V)HD\mid Y],\qquad
 g_x^{(s)}(Y)=\E[g(X,V)\mid Y]\,.
\end{equation*}
Conditional Cauchy--Schwarz, $g\le M$, and $\E_VH^2=6$ give
\begin{equation}\label{eq:commutator-energy}
 \int\frac{|e_s|^2}{g_x^{(s)}}\dd\pi_x^{(s)}
 \le\E[g(X,V)H^2|D|^2]
 \le6M\E|D|^2\le6ML^2ds\,.
\end{equation}
The last bound uses $\nabla U_s(Y)=\E[\nabla U(X)\mid Y]$ and Lipschitz continuity of~$\nabla U$. Cauchy--Schwarz and Talagrand's inequality~\cite{OttoVillani2000}, with constant~$m_s$ and $\Ent_x(g_x^{(s)})\le\log M$, now give
\begin{equation*}
 \left|\int\xi_{g_x^{(s)}}\cdot e_s\dd\pi_x^{(s)}\right|
 \le\sqrt{\frac{2\log M}{m_s}}\sqrt{6ML^2ds}\le C\sqrt{s}\,,
\end{equation*}
where $C$ depends only on $d,m,L,M$, uniformly in $t$, $\lambda>0$, and $0<s\le1$.

Combining the defect-term and commutator bounds as in the regular case gives, in distributions in time,
\begin{align*}
 \frac{\dd}{\dd t}\COT(g_t^{(s)})
 &\le-\frac12\Ent_x(g_{x,t}^{(s)})+A_r\Ent_v(g_t)-\gamma\COT(g_t^{(s)})\\
 &\quad+\eta_r+C\sqrt{s}\,.
\end{align*}

\emph{Removing the regularization.} Let $s\downarrow0$. The spatial entropy enters with a negative sign, so its lower semicontinuity under narrow convergence suffices, by Fatou's lemma against nonnegative time tests. Transport stability~\cite{LiLuHypercontractivity}*{Appendix A, Step 2} and uniqueness of the Brenier map identify every limit of the joint laws of $(X,V,T_{g_x^{(s)}}(X))$ under~$\mu_t*_x\zeta_s$ as the law of $(X,V,T_{g_x}(X))$ under~$\mu_t$. The density bound and pushforward identity give uniform fourth moments, hence uniform integrability of the current integrands. Thus $\COT(g_t^{(s)})\to\COT(g_t)$, with a bound uniform in $t$ and $s\le1$ by~\eqref{eq:current-size}. Passing to the limit in the weak formulation gives~\eqref{eq:closed-current}.

To remove the floor \(\lambda\), apply this result to the initial density \(\alpha+(1-\alpha)g_0\) with \(0<\alpha<1\). Since constant densities are stationary, the corresponding solution is \(\alpha+(1-\alpha)g_t\), with values in \([\alpha,M]\). Let \(\alpha\downarrow0\). The same limits apply, and convexity gives \(\Ent_v(\alpha+(1-\alpha)g_t)\le\Ent_v(g_t)\). Similar approximation arguments appear in~\cite{LiLuHypercontractivity}*{Appendix B} and~\cite{LuUnderdampedEntropy}*{Section 7}.

\emph{Continuity.} The map \(t\mapsto g_t\) is continuous in \(L^1(\pi)\), as verified in the regularization argument, and \(0\le g_t\le M\). Hence \(\Ent_\pi(g_t)\) is continuous by dominated convergence. The same stability argument for Brenier maps shows that \(\COT(g_t)\) is continuous.
\end{proof}

\bibliography{references}
\end{document}